\documentclass[a4,12pt,reqno]{amsart}
\usepackage{amsmath,amssymb, amsthm, a4wide} 
\usepackage{graphicx} 

\usepackage{url} \usepackage{multirow} 
\usepackage[font=footnotesize]{subfig} 
\usepackage{array}
\usepackage[colorlinks,citecolor=blue,urlcolor=blue]{hyperref}

\newcommand{\x}{\mathsf{x}} \newcommand{\W}{\mathsf{W}}
\newcommand{\R}{\mathbb{R}} 
\newcommand{\I}{\mathsf{I}}

\theoremstyle{plain} \newtheorem{theorem}{Theorem}
\newtheorem{proposition}{Proposition} \newtheorem{lemma}{Lemma}

\theoremstyle{definition} \newtheorem{assumption}{Assumption}

\allowdisplaybreaks[2]

\begin{document}

\begin{center}
{\Large
    {\sc Estimation of the Number of Spikes,\\
Possibly Equal, in the High-Dimensional Case}
}
\bigskip

Damien Passemier$^{1}$ \& Jianfeng Yao$^{2}$
\bigskip

{\it
$^{1}$Department of Electronic and Computer Engineering\\
Hong Kong University of Science and Technology (HKUST)\\
Clear Water Bay, Kowloon, Hong Kong\\
damien.passemier@gmail.com\\
\bigskip
$^{2}${Department of Statistics
and Actuarial Science\\ The University of Hong Kong\\ Pokfulam Road,
Hong Kong\\ jeffyao@hku.hk\\

}}
\end{center}
\bigskip

{\bf Abstract.} Estimating the number of spikes in a spiked model
is an important problem in many areas such as signal
processing. Most of the classical approaches assume a large sample
size $n$ whereas the dimension $p$ of the observations is kept
small. In this paper, we consider the case of high dimension, where
$p$ is large compared to $n$. The approach is based on recent results
of random matrix theory. We extend our previous results to a more
difficult situation where some spikes are equal, and compare our
algorithm to an existing benchmark method.

\smallskip


\section{Introduction}
The spiked population model has been introduced in \cite{Johnstone},
and appears in many scientific fields. In wireless communications, a
signal emitted by a source is modulated and received by an array of
antennas, and the reconstruction of the original signal is directly
linked to the inference of ``spikes''.  In psychology literature, the
strict factor model is equivalent to the spiked population model, and
the number of factors has a primary importance
\cite{Anderson}. Similar models can be found in physics of mixture
\cite{Nadler1}, \cite{Naes} or population genetics. More precisely, we
consider the following spiked population model for the observed
signals $\mathsf{x}(t)$:
\begin{eqnarray} \mathsf{x}(t)&=&\sum_{k=1}^{q_0} f_k(t)a_k + \sigma
n(t) = \mathsf{A} f(t)+ \sigma n(t)\text{,} \label{model}
\end{eqnarray} where $f(t)=(f_1(t),\ldots,f_{q_0}(t))^* \in
\mathbb{R}^{q_0}$ are $q_0$ independent random signals ($q_0 < p$)
with mean zero and unit variance; $\mathsf{A}=(a_1,\ldots,a_{q_0})$ is
a $p \times q_0$ full rank matrix (mixing weights); and $\sigma \in
\mathbb{R}$ is the unknown noise level, $n(t) \sim
\mathcal{N}(0,\mathsf{I}_p)$ is a $p$-dimensional Gaussian white noise
independent of $f(t)$.

The {\em population covariance matrix} $\Sigma$ of $\x(t)$ equals
$\mathsf{A}\mathsf{A}^* + \sigma^2 \I_p$ and has the spectral
decomposition
\[ \W^* \Sigma \W = \sigma^2 \I_p +
\mbox{diag}(\alpha_1,\dots,\alpha_{q_0},0,\dots,0) \] where $\W$ is an
unknown basis of $\R^p$ and $\alpha_1 \ge \alpha_2 \ge \dots \ge
\alpha_{q_0} > 0$. As in \cite{Passemier}, we rewrite the spectral
decomposition of $\Sigma$ as
\[ \W^* \Sigma \W =
\sigma^2\mbox{diag}(\alpha_1',\dots,\alpha_{q_0}',1,\dots,1)
\text{,}\] with $\alpha_i' = \alpha_i/\sigma^2 + 1$.

Notice that Model \eqref{model} is called a {\em strict factor model} 
in \cite{Anderson}. It should not be confused with those factor model widely used in econometric and finance which have generally a more complex structure. For instance, the term {\em strict factor model} is used in these fields to refer to a noise with a general diagonal covariance, while in the so-called approximate factor model or the dynamic factor model (see e.g. \cite{Chamberlain} and \cite{Forni}), the noise is cross-sectionally correlated unlike
the white noise structure assumed in \eqref{model}, and the {\em
factors} $f(t)$ could be a time-series unlike the i.i.d. structure
assumed in \eqref{model}. In sum, these factor models have a much more
complex structure than the spiked population model \eqref{model}
considered in this paper.

A fundamental inference problem in Model~\eqref{model} is the
determination of the number of spikes $q_0$. Many methods have been
developed, mostly based on information theoretic criteria, such as the
minimum description length (MDL) estimator, Bayesian model selection
or Bayesian Information Criteria (BIC) estimators, see \cite{Wax} for
a review. Nevertheless, these methods are based on asymptotic
expansions for large sample size and may not perform well when the
dimension of the data $p$ is large compared to the sample size $n$. To
our knowledge, this problem in the context of high-dimension appears
for the first time in \cite{Combettes}. Recent advances have been made
using the random matrix theory by \cite{Harding} or Onatski
\cite{Onatski} in economics, and Kritchman \& Nadler \cite{Nadler1} in
chemometric literature.

Several studies have also appeared in the area of signal processing
from high-dimensional data. Everson \& Roberts \cite{Everson} proposed
a method using both random matrix theory (RMT) and Bayesian inference,
while Ulfarsson \& Solo combined RMT and Stein's unbiased risk
estimator (SURE) in \cite{Solo}. In \cite{Nadakuditi} and
\cite{Nadler3}, the authors proposed some estimators using information
theoretic criteria.  Finally in \cite{Nadler2}, Kritchman \& Nadler
constructed an estimator based on the distribution of the largest
eigenvalue (hereafter refereed as the KN estimator). In
\cite{Passemier}, we have also introduced a new method based on recent
results of \cite{Bai-Yao} and \cite{Paul} in random matrix theory. It
is worth mentioning that for high-dimensional time series, an
empirical method for the estimation of the spike number has been
recently proposed in \cite{Qyao1} and \cite{Qyao2}.

In all the cited references above, spikes are assumed to be
distinct. However, we observe that when some of these spikes are close
each other, the estimation problem becomes more difficult and these
algorithms need to be modified. We refer this new situation as the
{\em case with possibly equal spikes} and its precise formulation will
be given in Section~\ref{equal}. The aim of this work is to extend our method
\cite{Passemier} to this new situation and to compare it with the KN
estimator, that is known in the literature as one of the best
estimation methods.

The rest of the paper is organized as follows. In
Section~\ref{results}, the estimation problem of the number of
possibly equal spikes is introduced, and our estimator is then
proposed with a proof for its asymptotic
consistency. Section~\ref{simul} provides simulation experiments to
assess the finite-sample quality of our estimator. In Section~\ref{C},
we analyze the influence of a tuning parameter $C$ used in our
procedure and propose an automatic calibration method of the
parameter. Next, we carry out simulation experiments in
Section~\ref{comp} to compare our method to the benchmark KN
estimator.  Conclusions then follow and the appendix collects all the
proofs.

\section{Main results}
\label{results} The sample covariance matrix of the $n$
$p$-dimensional i.i.d. vectors considered at each time $t$,
$(\x_i=\x(t_i))_{1\le i \le n}$ is
\[ \mathsf{S}_n= \frac{1}{n} \sum_{i=1}^n \mathsf{x}_i
\mathsf{x}_i^*\text{.} \] Denote by $\lambda_{n,1} \ge \lambda_{n,2}
\ge \dots \ge \lambda_{n,p}$ its eigenvalues. Our aim is to estimate
$q_0$ on the basis of $\mathsf{S}_n$. 
We first recall our previous result of \cite{Passemier} in the case of
different spikes. Next, we propose an extension of the algorithm to
the case with possibly equal spikes. The consistency of the extended
algorithm is established.

\subsection{Previous work: estimation with different spikes}

We consider the case where the $(\alpha_i)_{1\le i \le q_0}$ are all
different, so there are $q_0$ distinct spikes. It is assumed in the
sequel that $p$ and $n$ are related so that when $n \rightarrow
+\infty$, $p/n \rightarrow c > 0$. Therefore, $p$ can be large
compared to the sample size $n$ (high-dimensional case).

Moreover, we assumed that $\alpha_1' > \dots > \alpha_{q_0}' >
1+\sqrt{c}$ for all $i \in \{1,\dots,q_0\}$; i.e all the spikes
$\alpha_i$'s are greater than $\sigma^2\sqrt{c}$. For $\alpha \ne 1$,
we define the function
\[\phi(\alpha) = \alpha + \frac{c\alpha}{\alpha-1}\text{.}\] Baik and
Silverstein \cite{BS} proved that, under a moment condition on
$\mathsf{x}$, for each $k \in \{1,\dots,q_0\}$ and almost surely,
\[\lambda_{n,k}\longrightarrow \sigma^2\phi(\alpha_k')\text{.}\] They
also proved that for all $1\le i \le L$ with a prefixed range $L$ and
almost surely,
\[\lambda_{n,q_0+i} \rightarrow b=\sigma^2(1+\sqrt{c})^2 \text{.}\]
The estimation method of $q_0$ in \cite{Passemier} is based on a close
inspection of differences between consecutive eigenvalues
\[\delta_{n,j}=\lambda_{n,j}-\lambda_{n,j+1}\text{, } j \ge
1\text{.}\] Indeed, the results quoted above imply that
a.s. $\delta_{n,j} \rightarrow 0$, for $j \ge q_0$ whereas for $j <
q_0$, $\delta_{n,j}$ tends to a positive limit. Thus it becomes
possible to estimate $q_0$ from index-numbers $j$ where $\delta_{n,j}$
become small. More precisely, the estimator is
\begin{eqnarray} \hat{q}_n&=&\mbox{min}\{ j \in \{1,\dots,s\} :
\delta_{n,j+1} < d_n \}\text{,} \label{estimator}
\end{eqnarray} where $s >q_0$ is a fixed number big enough, and $d_n$
is a threshold to be defined. In practice, the integer $s$ should be
thought as a preliminary bound on the number of possible spikes. In
\cite{Passemier}, we proved the consistency of $\hat{q}_n$ providing
that the threshold satisfies $d_n \rightarrow 0$, $n^{2/3}d_n
\rightarrow +\infty$ and under the following assumption on the entries
of $\x$.

\begin{assumption} \label{hypothese} The entries $\mathsf{\x}^i$ of
the random vector $\mathsf{\x}$ have a symmetric law and a
sub-exponential decay, that means there exists positive constants $D$,
$D'$ such that, for all $t\ge D'$,
  $$\mathbb{P}(|\mathsf{\x}^i| \ge t^{D}) \le e^{-t}\text{.}$$
\end{assumption}

\subsection{Estimation with possibly equal spikes}
\label{equal}
As said in the Introduction, when some spikes are close each other,
estimation algorithms need to be modified. More precisely, we adopt
the following theoretic model with $K$ different spike strengths
$\alpha_1,\dots,\alpha_K$, each of them can appear $n_k$ times (equal
spikes), respectively. In other words,
\begin{eqnarray*} \mbox{spec}(\Sigma) &=& (
\underbrace{\alpha_1,\dots,\alpha_1}_{n_1},\dots,
\underbrace{\alpha_K,\dots,\alpha_K}_{n_K},\underbrace{0,\dots,0}_{p-q_0}
) + \sigma^2 (\underbrace{1,\dots,1}_p)\\ &=&\sigma^2(
\underbrace{\alpha_1',\dots,\alpha_1'}_{n_1},\dots,\underbrace{\alpha_K',\dots,\alpha_K'}_{n_K},\underbrace{1,\cdots,1}_{p-q_0}
) \text{.}
\end{eqnarray*} with $n_1+\dots+n_K=q_0$. When all the spikes are
unequal, differences between sample spike eigenvalues tend to a
positive constant, whereas with two equal spikes, such difference will
tend to zero. This fact creates an ambiguity with those differences
corresponding to the noise eigenvalues which also tend to
zero. However, the convergence of the $\delta_{n,i}$'s, for $i > q_0$
(noise) is faster (in $O_{\mathbb{P}} (n^{-2/3})$) than that of the
$\delta_{n,i}$ from equal spikes (in $O_{\mathbb{P}} ( n^{-1/2} )$) as
a consequence of Theorem 3.1 of Bai \& Yao \cite{Bai-Yao}. This is the
key feature we use to adapt the estimator (\ref{estimator}) to the
current situation with a new threshold $d_n$.  The precise asymptotic
consistency is as follows.

\begin{theorem} \label{consistence} Let $(\mathsf{x}_i)_{(1 \le i \le
n)}$ be $n$ copies i.i.d. of $\mathsf{x}$ which follows the model
(\ref{model}) and satisfies Assumption \ref{hypothese}. Suppose that
the population covariance matrix $\Sigma$ has $K$ non null and non
unit eigenvalues $\alpha_1>\dots>\alpha_{K}>\sigma^2\sqrt{c}$
with respective multiplicity $(n_k)_{1\le k \le K}$
($n_1+\dots+n_K=q_0$), and $p-q_0$ unit eigenvalues. Assume that
$\frac{p}{n} \rightarrow c >0$ when $n \rightarrow+\infty$. Let
$(d_n)_{n \ge 0}$ be a real sequence such that $d_n=o( n^{-1/2} )$ and
$n^{2/3}d_n \rightarrow +\infty$. Then the estimator $\hat{q}_n$ is
consistent; i.e $\widehat{q}_n \rightarrow q_0$ in probability when $n
\rightarrow +\infty$.
\end{theorem}

Compared to the previous situation, the only modification
of our estimator is a new condition $d_n=o( n^{-1/2} )$ on the
convergence rate of $d_n$. The proof of Theorem~1 is postponed to
Appendix.

There is a variation of the estimator defined as
follows.  Instead of making a decision once one difference $\delta_k$
is below the threshold $d_n$ (see (\ref{estimator})), one may decide
once two consecutive differences $\delta_k$ and $\delta_{k+1}$ are
both below $d_n$, i.e. define the estimator to be
  \begin{eqnarray} \hat{q}_n^*&=&\mbox{min}\{ j \in \{1,\dots,s\} :
\delta_{n,j+1} < d_n\text{ and }\delta_{n,j+2} < d_n\}
\text{.}\label{estimator2}
  \end{eqnarray} It can be easily checked that the proof for the
consistency of $\hat{q}_n$ applies equally to $\hat{q}_n^*$ under the
same conditions as in Theorem~\ref{consistence}.  This version of the
estimator will be used in all the simulation experiments
below. Intuitively, $\hat{q}_n^*$ should be more robust than
$\hat{q}_n$.  We notice that eventually more than two consecutive
differences could be used in \eqref{estimator2}. However, our
simulation experiments have shown that using more consecutive
differences does not improve significantly. So we limit ourselves the
simulation study to two consecutive differences.

\section{Implementation issues and overview of simulation experiments}
\label{simul}

The practical implementation of the estimator $\hat{q}_n^*$ depend on
two unknown parameters, namely the noise variance $\sigma^2$ and the
threshold sequence $d_n$.  As for an estimate of $\sigma^2$, we used
in \cite{Passemier} an algorithm based on the maximum likelihood
estimate
\[\widehat{\sigma}^2=\frac{1}{p-q_0} \sum_{i=q_0+1}^p
\lambda_{n,i}\text{.}\] As explained in \cite{Nadler1} and
\cite{Nadler2}, this estimator has a negative bias. Hence the authors
developed an improved estimator with a smaller bias. We will use this
improved estimator of the noise level when it is unknown.

It remains to choose a threshold sequence $d_n$.  As argued in
\cite{Passemier}, we use a sequence $d_n$ of the form
$Cn^{-2/3}\sqrt{2 \log \log n}$, where $C$ is a ``tuning" parameter to
be adjusted. In our Monte-Carlo experiments, we shall
consider two choices of $C$: the first one is manually tuned and used
to assess some theoretical properties of the estimator $\hat{q}_n^*$;
and the second one is a data-driven and automatically chosen value
that is used in real-life applications. This automatic choice is
introduced in Section~\ref{C}.

  In the remaining of the paper, we conduct extensive simulation
experiments to assess the quality of the estimator $\hat{q}_n^*$
including a detailed comparison with a benchmark detector from the
literature.

  In all experiments, data are generated with the assigned noise level
$\sigma^2=1$ and empirical values are calculated using 500 independent
replications. Table~\ref{tab:sum} gives a summary of the parameters
in the experiments. One should notice that both the given value of
$\sigma^2=1$ and the estimated one, as well as the manually tuned and
the automatic chosen values of $C$ are used in different scenarios.
There are in total three sets of experiments. The first set (Figures
1-2 and Models A, B), given in this section, illustrates the
convergence of our estimator in the situation of equal spikes. The
second set of experiments (Figures 3-4 and Models D-K) addresses the
performance of the automatic tuned $C$ and they are reported in
Section~\ref{C}.  The last set of experiments (Figures 5-6-7),
reported in Section~\ref{comp}, are designed for a comparison with the
benchmark detector KN.

\begin{table}[!ht]
  \caption{{\footnotesize Summary of parameters used in the simulation
experiments. (L: left, R: right)}}
  \begin{center} {\footnotesize
      \begin{tabular}{|c@{ }c@{ }c@{ }c@{ }c@{ }c@{ }c@{ }c|} \hline
Fig. & Mod. & spike & \multicolumn{4}{c}{Fixed parameters} & Var.\\
\cline{5-8} No. & No. & values & $p,n$ & $c$ & $\sigma^2$ & $C$ &
par.\\ \hline \hline \multirow{2}{*}{1} & &
\multirow{2}{*}{($\alpha$)} & $(200,800)$ & 1/4 &
\multirow{2}{*}{Given} & $5.5$ & \multirow{2}{*}{$\alpha$}\\ & & &
$(2000,500)$ & 4 & & 9 & \\ \hline \multirow{2}{*}{2} & A &
$(\alpha,\alpha,5)$ & $(200,800)$ & $1/4$ & \multirow{2}{*}{Given} &
$6$ & \multirow{2}{*}{$\alpha$}\\ & B & $(\alpha,\alpha,15)$ &
$(2000,500)$ & $4$ & & $9.9$ & \\ \hline \multirow{2}{*}{3L} & D &
$(6,5)$ & &\multirow{2}{*}{10} & \multirow{2}{*}{Given} &
\multirow{2}{*}{11 and auto} & \multirow{2}{*}{$n$}\\ & E & $(6,5,5)$
& & & & & \\ \hline \multirow{2}{*}{3R} & F & $(10,5)$ & &
\multirow{2}{*}{1} & \multirow{2}{*}{Given} & \multirow{2}{*}{5 and
auto} & \multirow{2}{*}{$n$}\\ & G &$(10,5,5)$ & & & & & \\ \hline

        \multirow{2}{*}{4L} & H & $(1.5)$ & &\multirow{2}{*}{1} &
\multirow{2}{*}{Given} & \multirow{2}{*}{5 and auto} &
\multirow{2}{*}{$n$}\\ & I & $(1.5,1.5)$ & & & & & \\ \hline
\multirow{2}{*}{4R} & J & $(2.5,1.5)$ & & \multirow{2}{*}{1} &
\multirow{2}{*}{Given} & \multirow{2}{*}{5 and auto} &
\multirow{2}{*}{$n$}\\ & K &$(2.5,1.5,1.5)$ & & & & & \\ \hline 5L & D
& $(6,5)$ & & 10 & Estimated & Auto & $n$\\ \hline 5R & J &
$(2.5,1.5)$ & & 1 & Estimated & Auto & $n$\\ \hline 6L & E & $(6,5,5)$
& & 10 & Estimated & Auto & $n$\\ \hline 6R & K & $(2.5,1.5,1.5)$ & &
1 & Estimated & Auto & $n$\\ \hline \multirow{2}{*}{7} &
\multirow{2}{*}{L} & \multirow{2}{*}{No spike} & & $1$ &
\multirow{2}{*}{Estimated} & \multirow{2}{*}{Auto} &
\multirow{2}{*}{$n$}\\ & & & & $10$ & & &\\ \hline
      \end{tabular}}
  \end{center}
  \label{tab:sum}
\end{table}

\subsection{Comparison between the case of equal spikes and different
spikes} In Figure \ref{fig:strength1}, we consider the case of a
single spike $\alpha$ and analyze the probability of
misestimation as a function of the value of  $\alpha$, for
$(p,n)=(200,800)$, $c=0.25$ and $(p,n)=(2000,500)$, $c=4$. We set
$C=5.5$ for the first case and $C=9$ for the second case (all manually
tuned). The noise
level $\sigma^2=1$ is given.
The estimator performs well: we recover the threshold from which the
behavior of the spike eigenvalues differ from the noise ones
($\sqrt{c}=0.5$ for the first case, and $2$ for the second).

\begin{figure}[!ht]
  \begin{center}
    \includegraphics{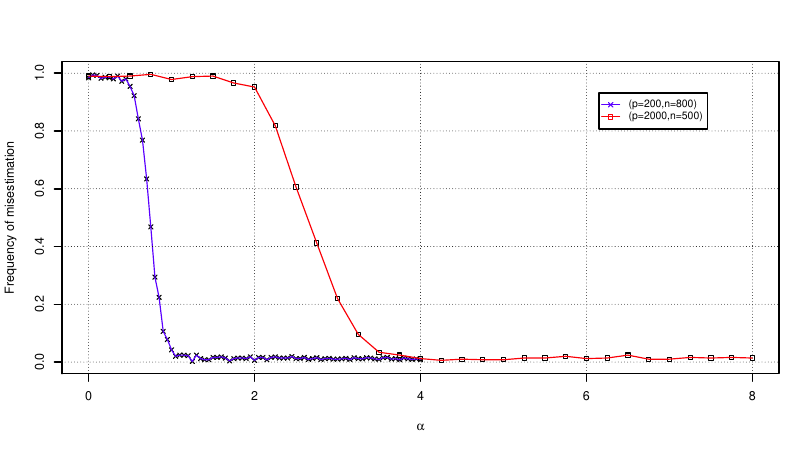}
  \end{center}
  \caption{{\footnotesize Misestimation rates as a function of spike
strength for $(p,n)=(200,800)$ and $(p,n)=(2000,500)$.}}
  \label{fig:strength1}
\end{figure}

Next we   keep the same parameters while adding some equal spikes. In
Figure \ref{fig:strength2}, we consider Model A:
$(\alpha_1,\alpha_2,\alpha_3)=(\alpha,\alpha,5)$, $0 \le \alpha \le
2.5$ and Model B: $(\alpha_1,\alpha_2,\alpha_3)=(\alpha,\alpha,15)$,
$0 \le \alpha \le 8$.  The dimensions are $(p,n)=(200,800)$ and $C=6$
for Model A, and $(p,n)=(2000,500)$ and $C=9.9$ for Model B.

\begin{figure}[!ht]
  \begin{center}
    \includegraphics{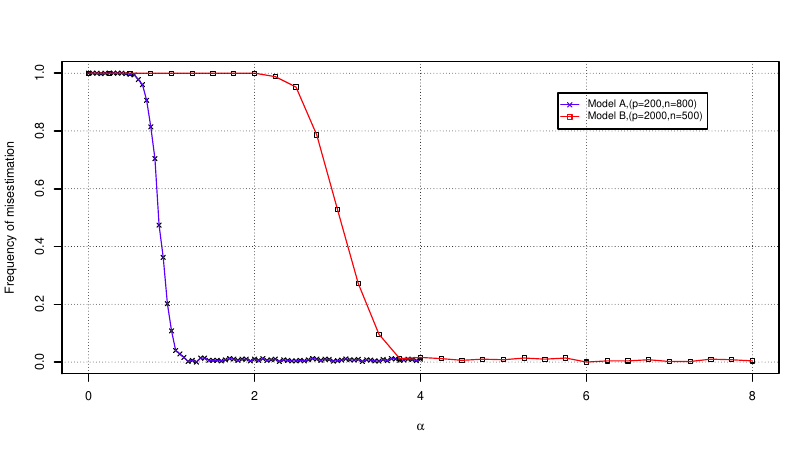}
  \end{center}
  \caption{{\footnotesize Misestimation rates as a function of spike
strength for $(p,n)=(200,800)$, Model A and $(p,n)=(2000,500)$, Model
B.}}
  \label{fig:strength2}
\end{figure}

There is no  particular difference with the previous situation:
when spikes are close 
or even equal, or near to the critical value, the estimator  remains
consistent although the convergence rate becomes slower. Overall,
these experiments demonstrate that the proposed estimator is able to
find the number of spikes.

\section{On the choice of $C$: an automatic calibration
procedure} \label{onc}
\label{C}

In the previous experiments, the tuning parameter $C$ has been
selected manually on a case by case basis. This is however untenable
in a real-life situation. We now provide an automatic calibration of
this parameter. The idea is to use the difference of the two largest
eigenvalues of a Wishart matrix (which correspond to the case of no
spike): indeed, our algorithm should stop once two consecutive
eigenvalues are below the threshold $d_n$ corresponding to a noise
eigenvalue. As we do not know precisely the distribution of the
difference between eigenvalues of a Wishart matrix, we approximate the
distribution of the difference between the two largest eigenvalues
$\tilde{\lambda}_1-\tilde{\lambda}_2$ by simulation under 500
independent replications. We then take the mean $s$ of the 10th and
the 11th largest spacings, so $s$ has the empirical probability
$\mathbb{P}(\tilde{\lambda}_1-\tilde{\lambda}_2 \le s)=0.98$: this
value will give reasonable results. We calculate a $\widetilde{C}$ by
multiplying this threshold by $n^{2/3}/\sqrt{2\times\log\log(n)}$. The
result for various $(p,n)$, with $c=1$ and $c=10$ are displayed in
Table \ref{tab:C}.

\begin{table}[!ht]
  \caption{{\footnotesize Approximation of the threshold $s$ such that
$\mathbb{P}(\tilde{\lambda}_1-\tilde{\lambda}_2 \le s)=0.98$.}}
  \begin{center} {\footnotesize
      \begin{tabular}{|c|c@{}c@{}c@{}|c@{}c@{}c@{}|} \hline (p,n) &
(200,200) & (400,400) & (600,600) & (2000,200) & (4000,400) &
(7000,700)\\ \hline Value of $s$ & 0.340 & 0.223 & 0.170 & 0.593 &
0.415 & 0.306 \\ \hline $\widetilde{C}$ & 6.367 & 6.398 & 6.277 & 11.106 &
11.906 & 12.44 \\ \hline
      \end{tabular}}
  \end{center}
  \label{tab:C}
\end{table} The values of $\widetilde{C}$ are quite close to the values we
would have chosen in similar settings (For instance, $C=5$ for $c=1$
and $C=9.9$ or 11 for $c=10$), although they are slightly
higher. Therefore, this automatic calibration of $\widetilde{C}$ can be
used in practice for any data and sample dimensions $p$ and $n$.

To assess the quality of the automatic calibration procedure, we run
some simulation experiments using both $\widetilde{C}$ and the manually
tuned $C$.  We consider in Figure \ref{fig:autoc1} the case
$c=10$. On the left panel we consider Model D ($\alpha=(6,5)$) and Model E
($\alpha=(6,5,5)$) (upper curve). On the right panel we have Model F
($\alpha=(10,5)$) and Model G ($\alpha=(10,5,5)$) (upper curve). The
dotted lines are the results with $C$ manually tuned.
\begin{figure}[!ht]
  \begin{center}
    \includegraphics{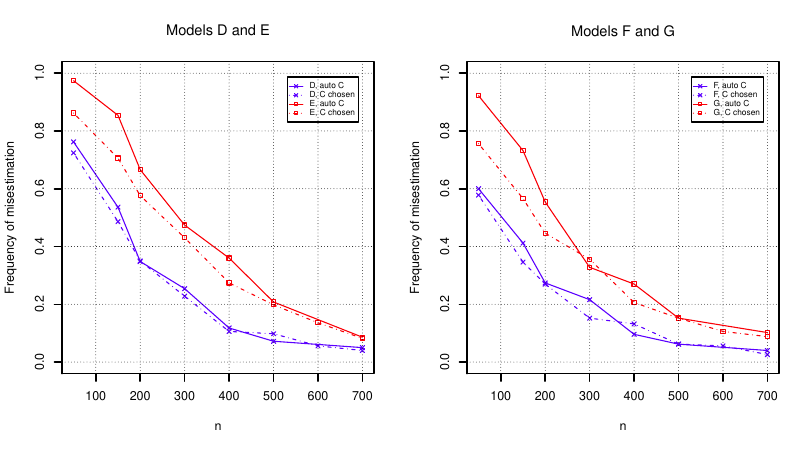}
  \end{center}
  \caption{{\footnotesize Misestimation rates as a function of $n$ for
Models D, E (left) and Models F, G (right).}}
  \label{fig:autoc1}
\end{figure}

Using the automatic value $\tilde C$ causes only a slight
deterioration of the estimation performance. We observe significantly
higher error rates in the case of equal spikes for moderate sample
sizes.

The case $c=1$ is considered in Figure \ref{fig:autoc2} with Models H
($\alpha=1.5$) and I ($\alpha=(1.5,1.5)$) (upper curve) on the left
and Model J ($\alpha=(2.5,1.5)$) and K ($\alpha=(2.5,1.5,1.5)$) (upper
curve) on the right.

\begin{figure}[!ht]
  \begin{center}
    \includegraphics{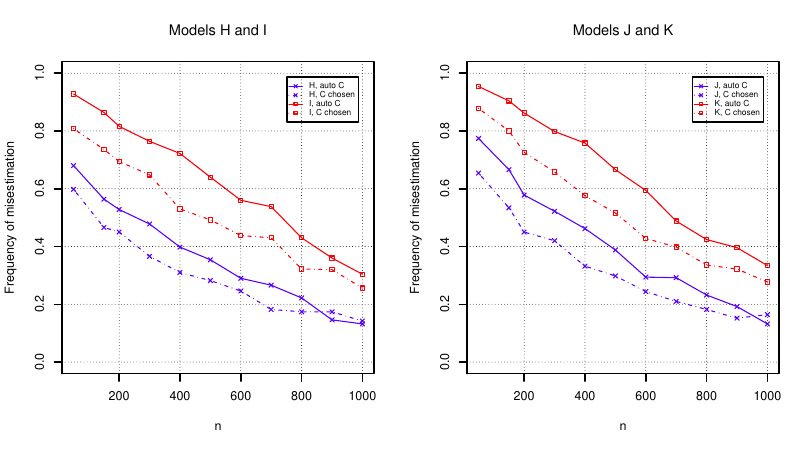}
  \end{center}
  \caption{{\footnotesize Misestimation rates as a function of $n$ for
Models H, I (left) and Models J, K (right).}}
  \label{fig:autoc2}
\end{figure}

Compared to the previous situation of $c=10$, using the automatic
value $\widetilde{C}$ affects a bit more our estimator (up to 20\% of
degradation). Nevertheless, the estimator remains
consistent. Furthermore, we have to keep in mind that our simulation
experiments have considered critical cases where spikes eigenvalues
are close: in many of practical applications, theses spikes are more
separated so that the influence of $C$ will be less important.

\section{Method of Kritchman \& Nadler and comparison}
\label{comp}
\subsection{Algorithm of Kritchman \& Nadler} In their papers
\cite{Nadler1} and \cite{Nadler2}, Kritchman \& Nadler develop a
different method to estimate the number of spikes. In this section we
compare by simulation our estimator (PY) with this method (KN). Notice
that these authors have compared their estimator KN with existing
estimators in the signal processing literature, based on the minimum
description length (MDL), Bayesian information criterion (BIC) and
Akaike information criterion (AIC) \cite{Wax}. In most of the studied
cases, the estimator KN performs better. Furthermore in
\cite{Nadler3}, this estimator is compared to an improved AIC
estimator and it still has a better performance. Thus we decide to
consider only this estimator KN for the comparison here.

In the absence of spikes, $n\mathsf{S}_n$ follows a Wishart
distribution with parameters $n,p$. In this case, Johnstone
\cite{Johnstone} has provided the asymptotic distribution of the
largest eigenvalue of $\mathsf{S}_n$.

\begin{proposition} Let $\mathsf{S}_n$ be the sample covariance matrix
of $n$ vectors distributed as $\mathcal{N}(0,\sigma^2 \mathsf{I}_p)$,
and $\lambda_{n,1} \ge \lambda_{n,2} \ge \dots \ge \lambda_{n,p}$ be
its eigenvalues. Then, when $n\rightarrow+\infty$, such that
$\frac{p}{n} \rightarrow c > 0$
  \[\mathbb{P} \left ( \frac{\lambda_{n,1}}{\sigma^2} <
\frac{\beta_{n,p}}{n^{2/3}}s+b \right ) \rightarrow F_1(s)\mbox{,
}s>0\] where $b=(1+\sqrt{c})^2$,
$\beta_{n,p}=\left(1+\sqrt{\frac{p}{n}}\right)\left(1+\sqrt{\frac{n}{p}}\right)^{\frac{1}{3}}$
and $F_1$ is the Tracy-Widom distribution of order 1.
\end{proposition}

Assume the variance $\sigma^2$ is known. To distinguish a spike
eigenvalue $\lambda$ from a noise one at an asymptotic significance
level $\gamma$, their idea is to check whether
\begin{eqnarray} \label{test} \lambda_{n,k} > \sigma^2 \left (
\frac{\beta_{n,p-k}}{n^{2/3}}s(\gamma)+b \right )
\end{eqnarray} where $s(\gamma)$ verifies $F_1(s(\gamma))=1-\gamma$
and can be found by inverting the Tracy-Widom distribution. This
distribution has no explicit expression, but can be computed from a
solution of a second order Painlev\'e ordinary differential
equation. The estimator KN is based on a sequence of nested hypothesis
tests of the following form: for $k=1,2,\ldots,\mbox{min}(p,n)-1$,
\[\mathcal{H}_0^{(k)}\mbox{: } q_0 \le k-1 \quad vs.\quad
\mathcal{H}_1^{(k)}\mbox{: }q_0 \ge k \mbox{ .}\] For each value of
$k$, if (\ref{test}) is satisfied, $\mathcal{H}_0^{(k)}$ is rejected
and $k$ is increased by one. The procedure stops once an instance of
$\mathcal{H}_0^{(k)}$ is accepted and the number of spikes is then
estimated to be $\tilde{q}_{n}=k-1$. Formally, their estimator is
defined by
\[\tilde{q}_{n}=\underset{k}{\mbox{argmin}} \left ( \lambda_{n,k} <
\widehat{\sigma}^2 \left (\frac{\beta_{n,p-k}}{n^{2/3}}s(\gamma)+b
\right ) \right ) -1\text{.}\] Here $\widehat{\sigma}$ is some
estimator of the noise level (as discussed in Section \ref{simul}). The
authors proved the strong consistency of their estimator as $n
\rightarrow +\infty$ with fixed $p$, by replacing the fixed confidence
level $\gamma$ with a sample-size dependent one $\gamma_n$, where
$\gamma_n \rightarrow 0$ sufficiently slow as $n \rightarrow
+\infty$. They also proved that $\lim_{p,n \rightarrow +\infty}
\mathbb{P} \left ( \tilde{q}_{n} \ge q_0 \right ) = 1$.

It is important to notice here that the construction of the KN
estimator differs form ours, essentially because of the fixed alarm
rate $\gamma$. We will discuss the issue of the false alarm rate in
the last section.

\subsection{Comparison with our method} In order to follow a real-life
situation, we assume that $\sigma^2=1$ is unknown and we estimate it
for both methods. Furthermore, we use the automatic calibration
procedure to choose the constant $C$ in our method. We give a value of
$\gamma=0.5\%$ to the false alarm rate of the estimator KN, as
suggested in \cite{Nadler2} and use their algorithm available at the
author's homepage.

We consider in Figure \ref{fig:compKN1} Model D:
$(\alpha_1,\alpha_2)=(6,5)$ and and Model J:
$(\alpha_1,\alpha_2)=(2.5,1.5)$.

\begin{figure}[!ht]
  \begin{center}
    \includegraphics{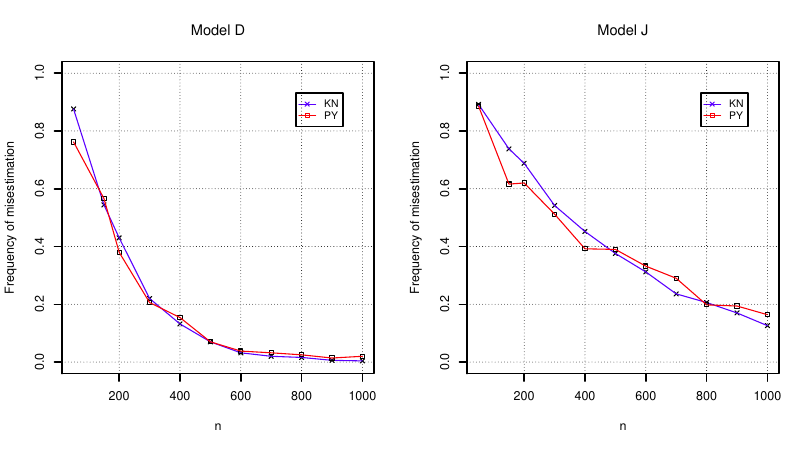}
  \end{center}
  \caption{{\footnotesize Misestimation rates as a function of $n$ for
Model D (left) and Model J (right).}}
  \label{fig:compKN1}
\end{figure}

For both Model D and J, the performances of the two estimator are
close. However the estimator PY is slightly better for moderate values
of $n$ ($n \le 400$) while the estimator KN has a slightly better
performance for larger $n$. The difference between the two estimators
are more important for Model J (up to 5\% improvement).

Next we consider in Figure \ref{fig:compKN2} Model E:
$(\alpha_1,\alpha_2,\alpha_2)=(6,5,5)$ and Model K:
$(\alpha_1,\alpha_2,\alpha_2)=(2.5,1.5,1.5)$, two models analogous to
Model D and J but with two equal spikes.

\begin{figure}[!ht]
  \begin{center}
    \includegraphics{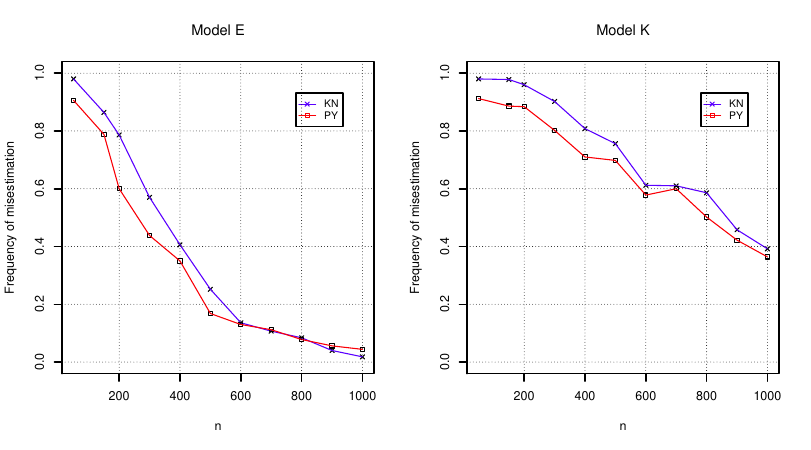}
  \end{center}
  \caption{{\footnotesize Misestimation rates as a function of $n$ for
Model E (left) and Model K (right).}}
  \label{fig:compKN2}
\end{figure}

For Model E, the estimator PY shows superior performance for $n \le
500$ (up to 20\% less error): adding an equal spike affects more the
performance of the estimator KN. The difference between the two
algorithms for Model K is higher than in the previous cases; the
estimator PY performs better in all cases, up to 10\%.
\newpage
In Figure \ref{fig:nospike} we examine a special case with no spike at
all (Model L). The estimation rates become the so-called false-alarm
rate, a concept widely used in signal processing literature. The cases
of $c=1$ and $c=10$ with $\sigma^2=1$ given are considered.

\begin{figure}[!ht]
  \begin{center}
    \includegraphics{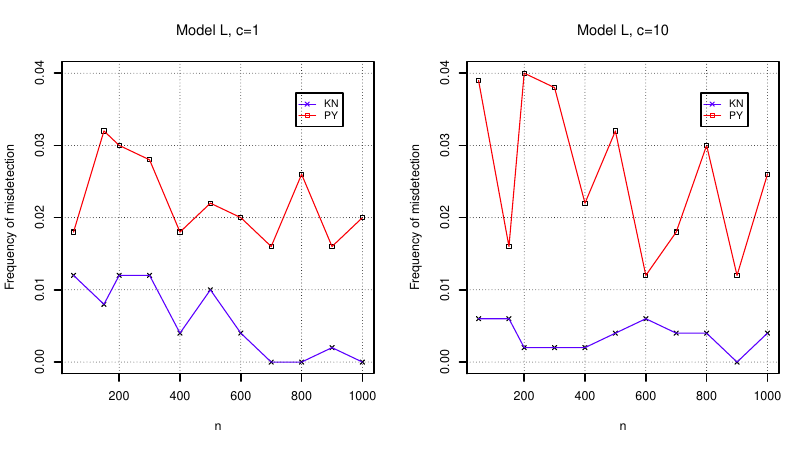}
  \end{center}
  \caption{{\footnotesize False-alarm rates as a function of $n$ for
$c=1$ (left) and $c=10$ (right).}}
  \label{fig:nospike}
\end{figure}

In both situations, the false-alarm rates of two estimators are quite
low (less than 4\%), and the KN one has a better performance.

In summary, in most of situations reported here, our algorithm
compares favorably to an existing benchmark method (the KN
estimator). It is also important to notice a fundamental difference
between these two estimators: the KN estimator is designed to keep the
false alarm rate at a very low level while our estimator attempts to
minimize an overall misestimation rate. We develop more in details
these issues in next section.

\subsection{Influence of $C$ on the misestimation and false alarm
rate} In the previous simulation experiments, we have chosen to report
the misestimation rates. However, to have a fair comparison to either
the KN estimator or any other method of the number of spikes, the
different methods should have comparable false alarm
probabilities. This section is devoted to an analysis of possible
relationship between the constant $C$ and the implied false alarm
rate.  Following \cite{Nadler2}, the false alarm rate $\gamma$ of such
an algorithm can be viewed as the type I error of the following test
\[\mathcal{H}_0\mbox{: } q_0 = 0 \quad vs.\quad \mathcal{H}_1\mbox{:
}q_0 > 0 \mbox{,}\] that is the probability of overestimation in the
white noise case. Recall the step $k$ of the algorithm KN tests
\[\mathcal{H}_0^{(k)}\mbox{: } q_0 \le k-1 \quad vs.\quad
\mathcal{H}_1^{(k)}\mbox{: }q_0 \ge k \mbox{.}\] In \cite{Nadler2},
the authors argue that their threshold is determined such that
\[\mathbb{P}(\mbox{reject } \mathcal{H}_0^{(k)} | \mathcal{H}_0^{(k)})
\approx \gamma\text{.}\] More precisely, they give an asymptotic bound
of the overestimation probability: they show that for $n=500$ and
$p>10$, this probability is close to $\gamma$.

Since for our method, we do not know explicitly the corresponding
false alarm rate, we recall in Table \ref{tab:far} the results from
Figure \ref{fig:nospike}, which were for the cases $c=1$ and $c=10$,
with the automatic value $\tilde C$, under 500 independent
replications.

\begin{table}[!ht]
  \caption{{\footnotesize False alarm rates in case of $c=1$ and
$c=10$.}}
  \begin{center} {\footnotesize
      \begin{tabular}{|c|c@{}c@{}c@{}c@{}|} \hline (p,n) & (150,150) &
(300,300) & (500,500) & (700,700) \\ \hline PY & 0.056 & 0.028 & 0.022
& 0.016 \\ KN & 0.008 & 0.002 & 0.001 & 0.000 \\ \hline \hline (p,n) &
(1500,150) & (3000,300) & (5000,500) & (7000,700) \\ \hline PY & 0.03
& 0.08 & 0.026 & 0.016 \\ KN & 0.004 & 0.006 & 0.004 & 0.006 \\ \hline
      \end{tabular}}
  \end{center}
  \label{tab:far}
\end{table}

As seen previously, the false alarm rates of our algorithm are much
higher than the false alarm rate $\gamma=0.5 \%$ of the KN
estimator. Nevertheless, and contrary to the KN estimator, the
overestimation rate of our estimator will be different from the false
alarm rate, and will depend on the number of spikes and their
values. Indeed, we use the gaps between two eigenvalues, instead of
each eigenvalue separately. Consequently, there is no justification to
claim that the probability $\mathbb{P}(\hat{q}_n > q | q=q_0)$, for
$q_0 >1$ will be close to $\mathbb{P}(\hat{q}_n > 0 | q=0)$. To
illustrate this phenomenon, we use the settings of Model E and K
($q_0=3$) and we evaluate the overestimation rate using 500
independent replications (note that the corresponding false alarm
rates are those in Table \ref{tab:far}). The results are displayed in
Table \ref{tab:over}.

\begin{table}[!ht]
  \caption{{\footnotesize Empirical overestimation rates from Model E
($\alpha=(6,5,5)$, $c=10$) and Model K ($\alpha=(2.5,1.5,1.5)$,
$c=1$).}}
  \begin{center} {\footnotesize
      \begin{tabular}{|c|c|cccc|} \cline{2-6} \multicolumn{1}{c|}{}
&(p,n) & (1500,150) & (3000,300) & (5000,500) & (7000,700) \\ \hline
\multirow{2}{*}{Model E} & PY & 0.004 & 0.022 & 0.018 & 0.01\\ & KN &
0.000 & 0.000 & 0.000 & 0.000\\ \hline \cline{2-6}
\multicolumn{1}{c|}{}&(p,n) & (150,150) & (300,300) & (500,500) &
(700,700) \\ \hline \multirow{2}{*}{Model K} & PY & 0.006 & 0.002 &
0.002 & 0.006 \\ & KN & 0.000 & 0.000 & 0.000 & 0.000 \\ \hline
      \end{tabular}}
  \end{center}
  \label{tab:over}
\end{table}

We observe that these overestimation rates are lower than the false
alarm rates given in Table 2: this suggests that no obvious
relationship exists between the false alarm rate $\gamma$ and the
overestimation rates for our algorithm.

Furthermore, the overestimation rates of the KN estimator are $0$ in
all cases: it means that misestimation is mostly underestimated.

In summary, if the goal is to keep overestimation rates at a constant
and low level, one should employ the KN estimator without hesitation
(since by construction, the probability of overestimation is kept to a
very low level). Otherwise, if the goal is also to minimize the
overall misestimation rates i.e. including underestimation errors, our
algorithm can be a good substitute to the KN estimator. One could
think of choosing $C$ in each case to have a probability of
overestimation kept fixed at a low level, but in this case the
probability of underestimation would be high and the performance of
the estimation would be poor: our estimator is constructed to minimize
the overall misestimation rate.

\section{Concluding remarks}
\label{concl}

In this paper we have considered the problem of the estimation of the
number of spikes in high-dimensional data. When some spikes have close
or even equal values, the estimation becomes harder and existing
algorithm need to be re-examined or corrected. In this spirit, we have
proposed a new version of our previous algorithm. Its asymptotic
consistency is established. We compare our algorithm to an existing
competitor proposed by Kritchman \& Nadler (KN, \cite{Nadler1},
\cite{Nadler2}). From our extensive simulation experiments in various
scenarios, overall our estimator can have smaller misestimation rates,
especially in cases with close and relatively low spike values (Figure
6) or more generally for almost all the cases provided that the sample
size $n$ is moderately large ($n\le 400$ or 500). Nevertheless, if the
primary aim is to fix the false alarm rate and the overestimation
rates at a very low level, the KN estimator is preferable.

Our algorithm depends on a tuning parameter $C$. By
comparison, the KN estimator is remarkably robust and a single value
of $\gamma=0.5\%$ was used in all the experiments. In Section
5 we have provided a first approach to an automatic calibration of $C$
which is quite satisfactory. More investigation is needed in the
future for this calibration in order to further improve the
performance of our estimator.

\noindent{\bf Acknowledgment}\quad The authors are grateful to two
Referees and the Associate Editor for their helpful comments that have
led to many improvements of the paper.

\bibliographystyle{alpha} 
\bibliography{ref}

\begin{thebibliography}{BGGM11}

\bibitem[And03]{Anderson}
T.~W. Anderson.
\newblock {\em An introduction to multivariate statistical analysis}.
\newblock Wiley Series in Probability and Statistics. Wiley-Interscience [John
  Wiley \& Sons], Hoboken, NJ, third edition, 2003.

\bibitem[BGGM11]{Maida}
F.~Benaych-Georges, A.~Guionnet, and M.~Maida.
\newblock Fluctuations of the extreme eigenvalues of finite rank deformations
  of random matrices.
\newblock {\em Electron. J. Probab.}, 16(60):1621--1662, 2011.

\bibitem[BS06]{BS}
Jinho Baik and Jack~W. Silverstein.
\newblock Eigenvalues of large sample covariance matrices of spiked population
  models.
\newblock {\em J. Multivariate Anal.}, 97(6):1382--1408, 2006.

\bibitem[BY08]{Bai-Yao}
Zhidong Bai and Jian-feng Yao.
\newblock Central limit theorems for eigenvalues in a spiked population model.
\newblock {\em Ann. Inst. Henri Poincar\'e Probab. Stat.}, 44(3):447--474,
  2008.

\bibitem[CR83]{Chamberlain}
Florent Chamberlain, Gary and Michael Rothschild.
\newblock Arbitrage, factor structure, and mean-variance analysis on large
  asset markets.
\newblock {\em Econometrica}, 51(5):1281--1304, 1983.

\bibitem[CS92]{Combettes}
L.~Combettes and Jack~W. Silverstein.
\newblock Signal detection via spectral theory of large dimensional random
  matrices.
\newblock {\em IEEE Trans. Signal Process.}, 40(8):2100--2105, 1992.

\bibitem[ER00]{Everson}
Richard Everson and Stephen Roberts.
\newblock Inferring the eigenvalues of covariance matrices from limited, noisy
  data.
\newblock {\em IEEE Trans. Signal Process.}, 48(7):2083--2091, 2000.

\bibitem[FHLR00]{Forni}
Mario Forni, Marc Hallin, Marco Lippi, and Lucrezia Reichlin.
\newblock The generalized dynamic factor model: identification and estimation.
\newblock {\em Rev. Econom. Statist.}, 82(4):540--554, 2000.

\bibitem[Har]{Harding}
M.C. Harding.
\newblock Estimating the number of factors in large dimensional factor models.
\newblock {\em Preprint}.

\bibitem[Joh01]{Johnstone}
Iain~M. Johnstone.
\newblock On the distribution of the largest eigenvalue in principal components
  analysis.
\newblock {\em Ann. Statist.}, 29(2):295--327, 2001.

\bibitem[KN08]{Nadler1}
Shira Kritchman and Boaz Nadler.
\newblock Determining the number of components in a factor model from limited
  noisy data.
\newblock {\em Chem. Int. Lab. Syst.}, 94, 2008.

\bibitem[KN09]{Nadler2}
Shira Kritchman and Boaz Nadler.
\newblock Non-parametric detection of the number of signals: {H}ypothesis
  testing and random matrix theory.
\newblock {\em IEEE Trans. Signal Process.}, 57(10):3930--3941, 2009.

\bibitem[LY12]{Qyao2}
C.~Lam and Q.~Yao.
\newblock Factor modeling for high-dimensional time series: inference for the
  number of factors.
\newblock {\em Ann. Statist.}, 40(2):694--726, 2012.

\bibitem[LYB11]{Qyao1}
C.~Lam, Q.~Yao, and N.~Bathia.
\newblock Estimation of latent factors for high-dimensional time series.
\newblock {\em Biometrika}, 98:901--918, 2011.

\bibitem[Nad]{Naes}
Boaz Nadler.
\newblock User-friendly guide to multivariate calibration and classification.
\newblock {\em NIR Publications, Chichester}.

\bibitem[Nad10]{Nadler3}
Boaz Nadler.
\newblock Nonparametric detection of signals by information theoretic criteria:
  performance analysis and an improved estimator.
\newblock {\em IEEE Trans. Signal Process.}, 58(5):2746--2756, 2010.

\bibitem[NE08]{Nadakuditi}
Raj~Rao Nadakuditi and Alan Edelman.
\newblock Sample eigenvalue based detection of high-dimensional signals in
  white noise using relatively few samples.
\newblock {\em IEEE Trans. Signal Process.}, 56(7, part 1):2625--2638, 2008.

\bibitem[Ona09]{Onatski}
Alexei Onatski.
\newblock Testing hypotheses about the numbers of factors in large factor
  models.
\newblock {\em Econometrica}, 77(5):1447--1479, 2009.

\bibitem[Pau07]{Paul}
Debashis Paul.
\newblock Asymptotics of sample eigenstructure for a large dimensional spiked
  covariance model.
\newblock {\em Statist. Sinica}, 17(4):1617--1642, 2007.

\bibitem[PY12]{Passemier}
Damien Passemier and Jian-feng Yao.
\newblock On determining the number of spikes in a high-dimensional spiked
  population model.
\newblock {\em Random Matrices: Theory and Applications}, 1(1):1150002, 2012.

\bibitem[US08]{Solo}
Magnus~O. Ulfarsson and Victor Solo.
\newblock Dimension estimation in noisy {PCA} with {SURE} and random matrix
  theory.
\newblock {\em IEEE Trans. Signal Process.}, 56(12):5804--5816, 2008.

\bibitem[WK85]{Wax}
Mati Wax and Thomas Kailath.
\newblock Detection of signals by information theoretic criteria.
\newblock {\em IEEE Trans. Acoust. Speech Signal Process.}, 33(2):387--392,
  1985.

\end{thebibliography}

\appendix
\section{Proof of Theorem~\ref{consistence}}
Without loss of generality we will assume that
$\sigma^2=1$ (if it is not the case, we consider
$\frac{\lambda_{n,j}}{\sigma^2}$). For the proof, we need the
following Propositions 2 and 3 from the literature. Proposition
\ref{Yao} is a result of Bai and Yao \cite{Bai-Yao} which derives from
a CLT for the $n_k$-packed eigenvalues
\[\sqrt{n}[\lambda_{n,j} - \phi(\alpha_k')]\mbox{, }j\in J_k\] where
$J_k=\{s_{k-1}+1,\dots,s_k\}$, $s_i=n_1+\cdots+n_i$ for $1\le i\le K$.
\begin{proposition} \label{Yao} Assume that the entries $\mathsf{x}^i$
of $\mathsf{x}$ satisfy $\mathbb{E}(\|\mathsf{x}^i\|^4) < +\infty$,
$\alpha_j'~>~1+\sqrt{c}$ for all $1 \le j \le K$ and have multiplicity
$n_1,\dots,n_K$ respectively. Then as $p$, $n \rightarrow +\infty$ so
that $\frac{p}{n} \rightarrow c$, the $n_k$-dimensional real vector
  \[\sqrt{n}\{\lambda_{n,j} - \phi(\alpha_k')\mbox{, }j \in J_k \}\]
converges weakly to the distribution of the $n_k$ eigenvalues of a
Gaussian random matrix whose covariance depends on $\alpha_k'$ and
$c$.
\end{proposition} Proposition \ref{Maida} is a direct consequence of
Proposition 5.8 from Benaych-Georges, Guionnet and Maida \cite{Maida}:
\begin{proposition} \label{Maida} Assume that the entries
$\mathsf{x}^i$ of $\mathsf{x}$ have a symmetric law and a
sub-exponential decay, that means there exists positive constants D,
D' such that, for all $t\ge \mbox{D'}$, we have
$\mathbb{P}(|\mathsf{x}^i|~\ge~t^{D})~\le~e^{-t}$. Then, for all $1
\le i \le L$ with a prefixed range $L$,
  \[n^{\frac{2}{3}}(\lambda_{n,q_0+i}-b) = O_{\mathbb{P}}(1)\text{.}\]
\end{proposition} We also need the following lemma:
\begin{lemma} \label{lemma1} Let $(\mathsf{X}_n)_{n \ge 0}$ be a
sequence of positive random variables which weakly converges to a
probability distribution with a continuous cumulative distribution
function. Then for all real sequence $(u_n)_{n \ge 0}$ which converges
to 0,
  \[\mathbb{P} (\mathsf{X}_n \le u_n) \rightarrow 0\text{.}\]
\end{lemma}

\begin{proof} As $(\mathsf{X}_n)_{n \ge 0}$ converges weakly, there
exists a function $G$ such that, for all $v\ge0$,
$\mathbb{P}(\mathsf{X}_n~\le~v) \rightarrow G(v)$.  Furthermore, as
$u_n \rightarrow 0$, there exists $N \in \mathbb{N}$ such that for all
$n~\ge~N$, $u_n~\le~v$. So $\mathbb{P}(\mathsf{X}_n \le u_n) \le
\mathbb{P}(\mathsf{X}_n \le v)$, and $\varlimsup\limits_{n \to
+\infty} \mathbb{P}(\mathsf{X}_n \le u_n) \le \varlimsup\limits_{n \to
+\infty} \mathbb{P}(\mathsf{X}_n~\le~v)=G(v)$. Now we can take $v
\rightarrow 0$: as $(\mathsf{X}_n)_{n \ge 0}$ is positive, $G(v)
\rightarrow 0$. Consequently, $\mathbb{P} (\mathsf{X}_n~\le~u_n)
\rightarrow~0\text{.}$
\end{proof}

\begin{proof} of Theorem \ref{consistence}.  The proof is essentially
the same as Theorem 3.1 in \cite{Passemier}, except when the spikes
are equal. We have
  \begin{align*} \{ \hat{q}_n = q_0 \}& = \{ q_0=\mbox{min}\{j :
\delta_{n,j+1} < d_n\} \} \\ & = \{ \forall j \in
\{1,\dots,q_0\}\mbox{, }\delta_{n,j} \ge d_n \} \cap
\{\delta_{n,q_0+1} < d_n\}\text{.}
  \end{align*} Therefore
  \begin{align*} \mathbb{P}(\hat{q}_n = q_0) & = \mathbb{P}\left (
\bigcap_{1\le j \le q_0} \{ \delta_{n,j} \ge d_n \} \cap
\{\delta_{n,q_0+1} < d_n\} \right )\\ & = 1 - \mathbb{P}\left (
\bigcup_{1\le j \le q_0} \{ \delta_{n,j} < d_n \} \cup
\{\delta_{n,q_0+1} \ge d_n\} \right )\\ & \ge 1 - \sum_{j=1}^{q_0}
\mathbb{P}(\delta_{n,j} < d_n)-\mathbb{P}(\delta_{n,q_0+1} \ge d_n )
\text{.}
  \end{align*}

  \noindent \emph{Case of $j=q_0+1$}.  In this case, $\delta_{n,q_0+1}
= \lambda_{n,q_0+1}-\lambda_{n,q_0+2}$ (noise eigenvalues). As $d_n
\rightarrow 0$ such that, $n^{2/3}d_n \rightarrow +\infty$, and by
using Proposition \ref{Maida} in the same manner as in the proof of
Theorem 3.1 in \cite{Passemier}, we have
  \[\mathbb{P}(\delta_{n,q_0+1} \ge d_n) \rightarrow 0\text{.}\]

  \vspace{0.2cm}
  \noindent \emph{Case of $1 \le j \le q_0$}. These indexes correspond
to the spike eigenvalues.

  \begin{list}{\labelitemi}{\leftmargin=1em}
  \item Let $I_1=\{1\le l \le q_0 | \mbox{card}(J_l)=1\}$ (simple
spike) and $I_2=\{l-1| l\in I_1 \text{ and } l-1 >1\}$. For all $j \in
I_1\cup I_2$, $\delta_{n,j}$ corresponds to a consecutive difference
of $\lambda_{n,j}$ issued from two different spikes, so we can still
use Proposition 3 and the proof of Theorem 3.1 in \cite{Passemier} to
show that
    \[\mathbb{P}( \delta_{n,j} < d_n) \rightarrow 0\mbox{, }\forall j
\in I_1\text{.}\]

  \item Let $I_3=\{1\le l \le q_0-1 | l \notin (I_1\cup I_2)\}$. For
all $j \in I_3$, it exists $k \in \{ 1,\dots,K \}$ such that $j \in
J_k$.
    \begin{list}{\labelitemi}{\leftmargin=2em}
    \item If $j+1 \in J_k$ then, by Proposition 2, $\mathsf{X}_n =
\sqrt{n}\delta_{n,j}$ converges weakly to a limit which has a density
function on $\R^+$. So by using Lemma \ref{lemma1} and that
$d_n=o(n^{-1/2})$, we have
      \[\mathbb{P} \left ( \delta_{n,j} < d_n \right ) = \mathbb{P}
\left ( \sqrt{n}\delta_{n,j} < \sqrt{n} d_n \right ) \rightarrow
0\text{;}\]
    \item Otherwise, $j+1 \notin J_k$, so $\alpha_j \ne
\alpha_{j+1}$. Consequently, as previously, $\delta_{n,j}$ corresponds
to a consecutive difference of $\lambda_{n,j}$ issued from two
different spikes, so we can still use Proposition 2 and the proof of
Theorem 3.1 in \cite{Passemier} to show that
      \[\mathbb{P}( \delta_{n,j} < d_n) \rightarrow 0\text{.}\]
    \end{list}
  \item The case of $j=q_0$ is considered as in \cite{Passemier}.
  \end{list} \vspace{0.2cm}

  \noindent \emph{Conclusion}.  $\mathbb{P}(\delta_{n,q_0+1} \ge d_n)
\rightarrow 0$ and $\sum_{j=1}^{q_0} \mathbb{P}( \delta_{n,j} < d_n)
\rightarrow 0$, therefore

  \[\mathbb{P}(\hat{q}_n = q_0) \underset{n \rightarrow
+\infty}{\longrightarrow}1\text{.}\]
\end{proof}

\end{document}